\documentclass[12pt,reqno]{amsart}
\usepackage{amsthm, amsmath, amssymb, amscd, latexsym, multicol, verbatim, enumerate, graphics}\usepackage{tabularx,colortbl}
\usepackage{hyperref}

\usepackage[top=1.15in, bottom=1.15in, left=1.17in, right=1.17in]{geometry}
\usepackage[english]{babel}                       

\usepackage{mathrsfs}
\usepackage[all]{xy}
\usepackage[dvips]{graphicx}

\newcounter{cnt}
 \makeatletter
\def\mydggeometry{\makeatletter\dg@YGRID=1\dg@XGRID=20\unitlength=0.003pt\makeatother}
\makeatother \theoremstyle{remark}
\numberwithin{equation}{section}

\theoremstyle{definition} 
\theoremstyle{definition}
\newtheorem{defn}{Definition}\theoremstyle{definition}

\newtheorem{theorem}{Theorem}
\newtheorem{lemma}{Lemma}

\newtheorem{remark}{Remark}
\newtheorem{example}{Example}
\newtheorem{conjecture}{Conjecture}
\newtheorem*{thm*}{Theorem}
\newtheorem*{thma*}{Theorem A}
\newtheorem*{thmb*}{Theorem B}

\newcommand{\ra}{\rightarrow}

\newcommand{\vp}{\varphi}

\newcommand{\s}{\sigma}
\newcommand{\ff}{\mathcal{F}}

\newcommand{\mc}{\mathbb{C}}
\newcommand{\bc}{\mathbb{C}}

\newcommand{\Fl}{\mathcal{F}l}
\newcommand{\lie}{\mathfrak}

\newcommand{\wt}{\widetilde}

\newcommand{\DC}{\operatorname*{DC}}
\newcommand{\SpDC}{\operatorname*{SpDC}}

\begin{document}

\author{Xin Fang, Ghislain Fourier}
\address{Xin Fang: Mathematisches Institut, Universit\"{a}t zu K\"{o}ln, Weyertal 86-90, D-50931, K\"{o}ln, Germany.}
\email{xinfang.math@gmail.com}
\address{Ghislain Fourier: Mathematisches Institut, Universit\"at Bonn}
\address{School of Mathematics and Statistics, University of Glasgow}
\email{ghislain.fourier@glasgow.ac.uk}

\date{}

\title[Torus fixed points and Genocchi numbers]{Torus fixed points in Schubert varieties and normalized median Genocchi numbers}

\begin{abstract}
We give a new proof for the fact that the number of torus fixed points for the degenerate flag variety is equal to the normalized median Genocchi number, using the identification with a certain Schubert variety. We further study the torus fixed points for the symplectic degenerate flag variety and develop a combinatorial model, symplectic Dellac configurations, so parametrize them. The number of these symplectic fixed points is conjectured to be the median Euler number.
\end{abstract}

\maketitle

\section*{Introduction}
We consider the Schubert variety $X_{\tau_n}$ associated to the Weyl group element
\[
{\tau_n} :=(s_ns_{n+1}\cdots s_{2n-2})\cdots (s_ks_{k+1}\cdots s_{2k-2})\cdots (s_3s_4)s_2\in  \mathfrak{S}_{2n} 
\]
in the partial flag variety $SL_{2n}/P$, where $P$ is the standard parabolic subalgebra associated to the simple roots $\{\alpha_1, \alpha_3, \ldots, \alpha_{2n-1} \}$. Then there is a natural action of a $2n-1$-dimensional torus $T_{2n-1}$ and we are mainly interested in the fixed points $X_{\tau_n}^{T_{2n-1}}$ of this torus action. It is well known that the fixed points are parametrized Weyl groups elements which are less or equal to $\tau_n$ in the Bruhat order (modulo the stabilizer of the parabolic, in this case, the subgroup generated by $s_1, s_3, \ldots, s_{2n-1}$). Our first result is
\begin{thma*} There is an explicit bijection $\mathbf{b}$ from Dellac configurations $\text{DC}_n$ (Definition~\ref{defn1}) of $2n$ columns and $n$ rows to $X_{\tau_n}^{T_{2n-1}}$, hence the number of torus fixed points is equal to the normalized median Genocchi number (see Section~\ref{Sec:1} for definition).
\end{thma*}
\noindent Here is a an example of the Dellac configuration corresponding to a fixed point for $n=3$:
\[\qquad \qquad \qquad  \qquad   \qquad  \qquad  
\begin{tabularx}{\textwidth}{|p{0.2cm}|p{0.2cm}|p{0.2cm}|p{0.2cm}|p{0.2cm}|p{0.2cm}|cc}
\cline{1-6}
  $\bullet$  	 &  $\bullet$ &	&   && &&\\
\cline{1-6}
	  & & & $\bullet$  & $\bullet$ & & $\mapsto$ & $\s=124536$\\
\cline{1-6}
  & & $\bullet$  &  & & $\bullet$ && \\
\cline{1-6}
\end{tabularx}
\]

\bigskip

We also consider Schubert varieties of the symplectic flag variety, e.g. the Schubert variety $X_{\overline{\tau}_{2n}}^{sp}$ corresponding to the element (of the symplectic Weyl group):
\[
\overline{\tau}_{2n} := (r_{2n}\cdots r_{n+1})\cdots (r_{2n}r_{2n-1}r_{2n-2})(r_{2n}r_{2n-1})r_{2n}(r_n\cdots r_{2n-2})\cdots (r_4r_5r_6)(r_3r_4)r_2 
\]
in the symplectic partial flag variety. In this case, there is a natural action of $T_{2n}$ on the Schubert variety and we are again interested in the fixed points of this torus action. To parametrize them similar to the non-symplectic case, we introduce symplectic Dellac configurations (Definition~\ref{defn2}). These are Dellac configurations with $4n$ columns and $2n$ rows, which are invariant under the involution mapping the $i$-th row to the $2n-i+1$-st row. Our second result is
\begin{thmb*} The torus fixed points in $X_{\overline{\tau}_{2n}}^{sp}$ are parametrized by the symplectic Dellac configurations $\SpDC_{2n}$.
\end{thmb*}
We conjecture that the number of symplectic Dellac configurations is equal to the normalized median Euler number (\cite{K97}).

\bigskip

We should explain here why we are interested in these particular Schubert varieties. E. Feigin (\cite{Fei11}) defined the degenerate flag variety
\[
\Fl^a_{n} := \{ (U_1, \ldots, U_{n-1}) \in \prod_{i=1}^{n-1} {\operatorname{Gr}}_i(\mathbb{C}^{n}) \mid \operatorname{pr}_{i+1} U_i \subset U_{i+1} \}
\]
where $\operatorname{pr}_{i}$ is the endomorphism of $\bc^{n}$ setting the $i$-th coordinate to be zero. This is in fact a flat degeneration of the classical flag variety $\Fl_n$, moreover it was shown in \cite{CFR12, CLL15} that there is an action of $T_{2n-1}$ on $\Fl^a_{n}$. The symplectic degenerate flag variety $(\Fl_{2n}^a)^{sp}$ has been defined in \cite{FFiL12} in a similar way.

\bigskip

The degenerate flag variety is one of the main objects in the  framework of PBW filtrations and degenerations on universal enveloping algebras of simple Lie algebras (see for various aspects \cite{FFoL11a, FFoL11b, FFoL13, FFR15, Hag14, Fou14, Fou15, CFR12}). Here, one obtains \textit{degenerate flag varieties} $\Fl^a(\lambda)$ as highest weight orbits of PBW degenerate modules. In \cite{Fei11, FFiL12} it has been shown that these highest weight orbits do have an interpretation as a variety of certain flags. 

\bigskip

Recently, it was shown in \cite{CL15} that these degenerate flag varieties are in fact our particular Schubert varieties:
\begin{thm*}(Cerulli Irelli-Lanini)
\begin{enumerate}
\item In the $\lie{sl}_n$-case, the degenerate flag variety $\Fl^a_{n}$ is isomorphic to the Schubert variety $X_{\tau_n}$, moreover the isomorphism $\zeta: \Fl^a_{n} \stackrel{\sim}{\longrightarrow}  X_{\tau_n} $ is $T_{2n-1}$-equivariant.
\item In the $\lie{sp}_{2n}$-case the degenerate symplectic flag variety is isomorphic to $X_{\overline{\tau}_{2n}}^{sp}$ and again the isomorphism $\zeta^{sp}: X_{\overline{\tau}_{2n}}^{sp} \stackrel{\sim}{\longrightarrow} (\Fl_{2n}^a)^{sp}$ is torus-equivariant.
\end{enumerate}
\end{thm*}

The torus fixed points of the degenerate flag variety in type $A_n$ have been studied in \cite{Fei11}.  In that paper, an explicit bijection $\mathbf{f}$  to the set of Dellac configurations has been provided. Hence it was shown that the number of torus fixed points is equal to the normalized median Genocchi number.
\par
Combining the theorem by Cerulli Irelli and Lanini with Theorem A, we obtain another proof of this fact, using the classical set up of Schubert varieties only. Moreover, we can show that the following diagram commutes (here $\alpha$ denotes the natural identification of $W_{\leq\tau_{n}}^J $ with $X_{\tau_n}^{T_{2n-1}}$)
\[
\xymatrix{
(\ff l_{n}^a)^{T_{n}} \ar[r]^-{ \mathbf{f}} \ar[d]^-{\zeta} & {\DC}_{n} \ar[d]^-{\mathbf{b}}\\ X_{\tau_n}^{T_{2n-1}}& W_{\leq\tau_{n}}^J \ar[l]_-{\alpha}}.
\]
In the symplectic case, the map $\mathbf{f}$ is not present, mainly because the construction of symplectic Dellac configurations has not been seen in the literature before. Nevertheless we obtain a similar picture, namely the number of torus fixed points in the symplectic degenerate flag variety are parametrized by $\SpDC_{2n}$. We should mention here that E. Feigin (via the symplectic degenerate flag variety \cite{FFiL12}) as well as G. Cerulli Irelli (via quiver Grassmannian \cite{CFR12}) also conjectured the number of torus fixed points to be the normalized median Euler number. 

\bigskip

This paper is organized as follow, in Section~\ref{Sec:1} we prove our first theorem for the $\lie{sl}_n$, in Section~\ref{Sec:2} we consider the symplectic case. In Section~\ref{Sec:3} we relate our results to the framework of degenerate flag varieties.

\bigskip

\textbf{Acknowledgments}
The work of Xin Fang is supported by the Alexander von Humboldt Foundation. The work of Ghislain Fourier is funded by the DFG priority program 1388 ''Representation Theory''. The authors would like to thank Evgeny Feigin and Bruce Sagan for their helpful comments.

\section{Symmetric groups and Median Genocchi numbers}\label{Sec:1}

\subsection{}
Let $W=\mathfrak{S}_{2n}$ be the symmetric group generated by $S=\{s_1, s_2,\cdots, s_{2n-1}\}$ where $s_i=(i,i+1)$. Let $J=\{s_1,s_3,\cdots,s_{2n-1}\}\subset S$ and $W_J$ be the subgroup generated by $J$, $W^J$ be the set of minimal representatives of right cosets of $W_J$ in $W$. We define 
$${\tau_n}=(s_ns_{n+1}\cdots s_{2n-2})\cdots (s_ks_{k+1}\cdots s_{2k-2})\cdots (s_3s_4)s_2\in W,$$
then for $t=1,2,\cdots,2n$:
\begin{equation}\label{Eq:tau}
{\tau_n}(t)=\left\{\begin{matrix} k,& t=2k-1;\\
n+k, & t=2k.\end{matrix}\right.
\end{equation}
By construction, ${\tau_n}$ is a representative of minimal length in $W/W_J$, so ${\tau_n}\in W^J$. We define 
$$W_{\leq {\tau_n}}=\{w\in W|\ w\leq{\tau_n}\},\ \ W_{\leq {\tau_n}}^J=\{w\in W^J|\ w\leq{\tau_n}\},$$
where $\leq$ is the Bruhat order.
\par
\begin{defn}\label{defn1}
A Dellac configuration $C$ is a board of $2n$ columns and $n$ rows with $2n$ marked cells such that
\begin{enumerate}
\item each column contains exactly one marked cell;
\item each row contains exactly two marked cells;
\item if the $(i,j)$-cell is marked, then $i\leq j\leq n+i$.
\end{enumerate}
Let $\DC_n$ denote the set of such configurations.
\end{defn} 
It is worthy of pointing out that the definition of a Dellac configuration given above differs from that in \cite{Fei11} by rotating the board by $90^\circ$.
\par
The cardinality $h_n$ of the set $\DC_n$ is called a normalized median Genocchi number (see \cite{Fei11, Fei12} and the references therein). Consider the following polynomial defined by recursion:  $H_0(x)=1$,
$$H_n(x)=\frac{1}{2}(x+1)((x+1)H_{n-1}(x+1)-xH_{n-1}(x)).$$
Then it is proved in \cite{DR94} that $h_n=H_n(1)$.
\par
The following theorem is originally proved by Cerulli Irelli and Lanini in \cite{CL15} as a corollary of their main result and a result of Feigin \cite{Fei11} (see Remark~\ref{Rmk:proof} for details).
\begin{theorem}\label{Thm:Ancount} 
For any integer $n\geq 1$, $h_n=\# W_{\leq {\tau_n}}^J$.
\end{theorem}

We provide in this section a purely combinatorial bijective proof of the theorem.

\subsection{Rook arrangements}\label{Sec:Rook}
Consider a board of $n$ rows and columns. A rook arrangement $R$ is a filling of the cells by $n$ marks such that each row and each column have exactly one mark. Let $\mathcal{R}_n$ denote the set of all rook arrangements. There is a bijection 
\begin{align}\label{phi}
\vp:\mathcal{R}_n\stackrel{\sim}{\longrightarrow}\mathfrak{S}_n
\end{align}
sending a rook arrangement $R$ to the permutation $\s_{R}$ satisfying: for $i=1,\cdots,n$, $\s_R(i)=j$ if and only if the cell $(i,j)$ is marked in $R$. For $\s\in\mathfrak{S}_n$, we denote $R_\s:=\vp^{-1}(\s)$.
\par
Let $R$ be a rook arrangement. The convex hull of the marked cells in $R$ is the smallest right-aligned skew-Ferrers board containing all marks in $R$.
\par
From now on we consider $\mathfrak{S}_{2n}$: $R_{\tau_n}$ is a board of $2n$ columns and rows. A restricted rook arrangement with respect to ${\tau_n}$ is a rook arrangement such that all marked cells in the board are contained in the convex hull (it is called the right hull in \cite{Sjo07}) of the marked cells in $R_{\tau_n}$. Let $R_{\leq {\tau_n}}$ denote the set of all restricted rook arrangements with respect to ${\tau_n}$. 

\begin{example}\label{Ex:1}
We consider an example where $n=3$, then $\tau_3 = 142536$ and the shadowed area is the called the convex hull of the marked cells in $R_{\tau_3}$.  We fix $\s=124536$, then the rook arrangement of $\s$ is (given by the dots):\\
\[ \qquad \qquad \qquad  \qquad \qquad \qquad 
R_\s=\begin{tabularx}{\textwidth}{|p{0.2cm}|p{0.2cm}|p{0.2cm}|p{0.2cm}|p{0.2cm}|p{0.2cm}|}
\cline{1-6}
  $\bullet$  \cellcolor[gray]{.8}	 & &	&   && \\
\cline{1-6}
	  & $\bullet$ \cellcolor[gray]{.8} &  \cellcolor[gray]{.8} &  \cellcolor[gray]{.8} & &  \\
\cline{1-6}
   & \cellcolor[gray]{.8} &   \cellcolor[gray]{.8}& $\bullet$\cellcolor[gray]{.8} & & \\
\cline{1-6}
   & & \cellcolor[gray]{.8} & \cellcolor[gray]{.8} & $\bullet$  \cellcolor[gray]{.8}   &\\ 
\cline{1-6}
  &  &  $\bullet$\cellcolor[gray]{.8} & \cellcolor[gray]{.8} & \cellcolor[gray]{.8}  &\\
\cline{1-6}
   & & & & &$\bullet$  \cellcolor[gray]{.8}   \\
\cline{1-6}
\end{tabularx}
\]
$ $\\
$R_\s$ is the restricted rook arrangement with respect to $\tau_3$.
\end{example}

It is clear that ${\tau_n}$ avoids the patterns $4231$, $35142$, $42513$, and $351624$. The following result is a special case of Theorem 4 in \cite{Sjo07}.
\begin{theorem}[\cite{Sjo07}]\label{Thm:bijection}
The restriction of $\vp$ on $R_{\leq {\tau_n}}$ gives a bijection $R_{\leq {\tau_n}} \stackrel{\sim}{\longrightarrow} W_{\leq{\tau_n}}$.
\end{theorem}

\subsection{From rook arrangements to Dellac configurations}
We define two maps $\mathbf{m}:R_{\leq {\tau_n}}\ra \DC_n$ called the melt map and $\mathbf{b}:\DC_n\ra R_{\leq{\tau_n}}$ called the blow map.

\par
Let $R\in R_{\leq {\tau_n}}$ be a restricted rook arrangement. Consider a board $C_R$ of $2n$ columns and $n$ rows defined by: the cell $(k,l)$ of $C_R$ is marked if and only if either the cell $(2k-1,l)$ or the cell $(2k,l)$ is marked in $R$. Intuitively, the $k$-th row of $C_R$ is obtained by merging the $(2k-1)$-th and the $2k$-th rows in $R$.

\begin{lemma}
The board $C_R$ is a Dellac configuration.
\end{lemma}

\begin{proof}
By the definition of a rook arrangement, each row of $C_R$ has exactly two marked cells; each column of $C_R$ has exactly one marked cell. When moreover $R$ is restricted with respect to ${\tau_n}$, by (\ref{Eq:tau}), $C_R$ has the following property: if the cell $(r,s)$ in $C_R$ is marked, then $r\leq s\leq n+r$.
\end{proof}

By using the lemma we obtain a well-defined melt map
\[
\mathbf{m}(R):=C_R.
\]

\par

Let $C\in \DC_n$ be a Dellac configuration. A board $R_C$ of $2n$ rows and columns is associated to $C$ in the following way: the cells $(i,j)$ and $(i,k)$ with $j<k$ are marked in $C$ if and only if the cells $(2i-1,j)$ and $(2i,k)$ are marked in $R_C$. Intuitively, the $i$-th row in $C$ is splitted into two rows where the first row bears the first marked point and the second row admits the second one.

\begin{example}
Let $\s=124536$ be the permutation in Example~\ref{Ex:1}. The corresponding Dellac configuration via the melt procedure is given by:
\[\qquad \qquad \qquad  \qquad \qquad \qquad 
\begin{tabularx}{\textwidth}{|p{0.2cm}|p{0.2cm}|p{0.2cm}|p{0.2cm}|p{0.2cm}|p{0.2cm}|}
\cline{1-6}
  $\bullet$  	 &  $\bullet$ &	&   && \\
\cline{1-6}
	  & & & $\bullet$  & $\bullet$ &  \\
\cline{1-6}
   & & $\bullet$  &  & & $\bullet$\\
\cline{1-6}
\end{tabularx}
\]

\bigskip

\end{example}

\begin{lemma}\label{Lem:res}
The board $R_C$ is a restricted rook arrangement with respect to ${\tau_n}$.
\end{lemma}

\begin{proof}
Conditions (1) and (2) in the definition of the Dellac configuration guarantees that $R_C$ is a rook arrangement. The condition (3) means that $R_C$ is restricted with respect to ${\tau_n}$.
\end{proof}

By defining $\mathbf{b}(C)=R_C$, the blow map is well-defined by Lemma~\ref{Lem:res}.

\begin{lemma}\label{Lem:bij}
The following statements hold:
\begin{enumerate}
\item the map $\mathbf{b}$ is injective with $\text{im}(\mathbf{b})=\vp^{-1}(W_{\leq{\tau_n}}^J)$;
\item we have $\mathbf{m}\circ \mathbf{b}=\text{id}$.
\end{enumerate}
\end{lemma}

\begin{proof}
By construction, the only thing to be prove is $\text{im}(\mathbf{b})=\vp^{-1}(W_{\leq{\tau_n}}^J)$. It holds by the following description of $W^J$:
$$W^J=\{\s\in W|\ \text{$\s(2k-1)<\s(2k)$ for any $1\leq k\leq n$}\}.$$
\end{proof}

As an application of these maps, we give a bijective proof of Theorem~\ref{Thm:Ancount}:

\begin{proof}[Proof of Theorem~\ref{Thm:Ancount}]
By Lemma~\ref{Lem:bij}, the blow map $\mathbf{b}$ induces a bijection $\DC_n\stackrel{\sim}{\longrightarrow} W_{\leq{\tau_n}}^J$. By counting numbers we proved $h_n=\# W_{\leq {\tau_n}}^J$.
\end{proof}

\begin{remark}
The normalized median Genocchi numbers $h_n$ count a combinatorial structure in $\mathfrak{S}_{2n+2}$ called normalized Dumont permutation. Although \emph{a posteriori} there exists a bijection between the normalized Dumont permutation and $W_{\leq\tau_n}^J$, our approach is different from the one in \cite{K97}, see also \cite{Fei11}.
\end{remark}

\section{Symplectic case}\label{Sec:2}

\subsection{Notations}
Let $\wt{W}=\mathfrak{S}_{4n}$ be the symmetric group, $\wt{J}=\{s_1,s_3,\cdots,s_{4n-1}\}$. Let $\iota$ be the involution of $\wt{W}$ defined by: 
$$
\iota(\s)(k)=4n+1-\s(4n+1-k) \text { for }  \s\in \wt{W} \text{ and } 1\leq k\leq 4n.
$$
The Weyl group $W$ of the symplectic group $\text{Sp}_{4n}$ with generators $\{r_1,r_2,\cdots,r_{2n}\}$ can be embedded into $\wt{W}$ via the map $\kappa: W\ra \wt{W}$, $r_i\mapsto s_is_{4n-i}$ for $1\leq i\leq 2n-1$ and $r_{2n}\mapsto s_{2n}$. The image of $\kappa$ are the $\iota$-fixed elements $\wt{W}^\iota$ in $W$. Let $J=\{r_1,r_3,\cdots,r_{2n-1}\}$. We denote
\footnotesize
$$
\overline{\tau}_{2n}=(r_{2n}\cdots r_{n+1})\cdots (r_{2n}r_{2n-1}r_{2n-2})(r_{2n}r_{2n-1})r_{2n}(r_n\cdots r_{2n-2})\cdots (r_4r_5r_6)(r_3r_4)r_2\in W.
$$
\normalsize
It is observed in \cite{CLL15} that $\kappa(\overline{\tau}_{2n})=\tau_{2n}$.
\par
By Corollary 8.1.9 in \cite{GTM05} (notice the differences between the indices here and those in the reference), the restriction of $\kappa$ to $W_{\leq\overline{\tau}_{2n}}$ gives a bijection 
$$
\alpha:W_{\leq\overline{\tau}_{2n}}\stackrel{\sim}{\longrightarrow}  (\wt{W}_{\leq \tau_{2n}})^\iota.
$$ 
By passing to the right cosets, $\alpha$ induces a bijection $\alpha': W_{\leq\overline{\tau}_{2n}}^J\stackrel{\sim}{\longrightarrow}  (\wt{W}_{\leq \tau_{2n}}^{\wt{J}})^\iota$.

\subsection{Symplectic Dellac configurations}
\begin{defn}\label{defn2} A symplectic Dellac configuration $C$ is a board of $4n$ columns and $2n$ rows with $4n$ marked cells such that
\begin{enumerate}
\item each column contains exactly one marked cell;
\item each row contains exactly two marked cells;
\item if the $(i,j)$-cell is marked, then $i\leq j\leq 2n+i$;
\item for $1\leq i,j\leq 2n$, the $(i,j)$-cell is marked if and only if the $(2n-i+1,4n-j+1)$-cell is marked.
\end{enumerate}
Let $\SpDC_{2n}$ denote the set of such configurations and $e_n$ its cardinality. 
\end{defn}
We have  $e_1 = 1, e_2 = 2,e_3 = 10,e_4 = 98, e_5 = 1594$. Consider the sequence of polynomials defined by recursion: $E_0(x)=1$,
$$E_n(x)=\frac{1}{2}(x+1)((x+2)E_{n-1}(x+2)-xE_{n-1}(x)).$$
\begin{conjecture}
For any $n\geq 0$, $e_{n+1}=E_n(1)$.
\end{conjecture}
\begin{remark}
Giovanni Cerulli Irelli and Evgeny Feigin kindly informed us that they have also a similar conjecture.
\end{remark}
If this conjecture were true, these numbers $e_n$ coincide with the numbers $r_n$ in \cite{RZ96} (see A098279 in OEIS), where their continued fraction developments are studied (Th\'eor\`eme 29 in \emph{loc. cit.}).

\subsection{Main result}
The main result of this section is the following

\begin{theorem}\label{Thm:Cncount}
For any integer $n\geq 1$, $e_n=\# W_{\leq\overline{\tau}_{2n}}^J$.
\end{theorem}

\begin{proof}
We prove the theorem by establishing a bijection between $W_{\leq\overline{\tau}_{2n}}^J$ and $\SpDC_{2n}$, following the strategy in the proof of Theorem~\ref{Thm:Ancount}. 
\par
A symplectic rook arrangement $C$ is a board of $4n$ columns and rows with $4n$ marked points satisfying:
\begin{enumerate}
\item $C$ is a rook arrangement;
\item for any $1\leq i\leq 4n$ and $1\leq j\leq 2n$, the cell $(i,j)$ is marked if and only if the cell $(4n+1-i, 4n+1-j)$ is marked.
\end{enumerate}
The set of symplectic rook arrangements is denoted by $\mathcal{SR}_{4n}$. Similarly to Section~\ref{Sec:Rook} we can define the restricted symplectic rook arrangements with respect to $\tau_{2n}$: $\mathcal{SR}_{\leq \tau_{2n}}:=\mathcal{SR}_{4n}\cap \mathcal{R}_{\leq\tau_{2n}}$. 
\par
Consider the bijection $\vp: \mathcal{R}_{4n}\stackrel{\sim}{\longrightarrow}\mathfrak{S}_{4n}$ from (\ref{phi}).
\begin{lemma}
\begin{enumerate}
\item The restriction of the map $\vp$ induces a bijection $\vp':\mathcal{SR}_{4n}\stackrel{\sim}{\longrightarrow} \wt{W}^\iota=W$. 
\item The restriction of the map $\vp'$ induces a bijection $\psi:\mathcal{SR}_{\leq\tau_{2n}}\stackrel{\sim}{\longrightarrow} (\wt{W}_{\leq\tau_{2n}})^\iota$.
\end{enumerate}
\end{lemma}

\begin{proof}
\begin{enumerate}
\item Take a board $R$ in $\mathcal{SR}_{4n}$, the condition (2) in its definition implies that $\vp(R)$ is invariant under the involution $\iota$. It suffices to show that $\vp'$ is surjective: let $\s\in\wt{W}$, by definition of $\iota$, $\s$ is fixed by the involution $\iota$ if and only if $\s(4n+1-k)=4n+1-\s(k)$ for any $1\leq k\leq 4n$, i.e., for any $1\leq i\leq 4n$ and $1\leq j\leq 2n$, $\s(i)=j$ if and only if $\s(4n+1-i)=4n+1-j$. It implies that $\vp^{-1}(\s)$ is in $\mathcal{SR}_{4n}$.
\item Since $\mathcal{SR}_{\leq \tau_{2n}}=\mathcal{SR}_{4n}\cap \mathcal{R}_{\leq\tau_{2n}}$ and $(\wt{W}_{\leq\tau_{2n}})^\iota=\wt{W}^\iota\cap \wt{W}_{\leq \tau_{2n}}$, the bijectivity of $\psi$ follows from (1) and Theorem~\ref{Thm:bijection}.
\end{enumerate}
\end{proof}
\par
Moreover, consider the restriction of the melt map $\mathbf{m}:\mathcal{R}_{\leq \tau_{2n}}\ra\DC_{2n}$ on $\mathcal{SR}_{\leq \tau_{2n}}$. Since the condition (2) in the definition of the symplectic rook arrangement translates to the condition (4) in the definition of the symplectic Dellac configuration under the melt map, $\mathbf{m}$ induces a map $\mathbf{m}':\mathcal{SR}_{\leq\tau_{2n}}\ra\SpDC_{2n}$.
\par
\begin{example}
Let us consider an example where $n=2$ and the permutation is giving by the following rook arrangement:
$$\qquad \qquad \qquad  \qquad \qquad \qquad 
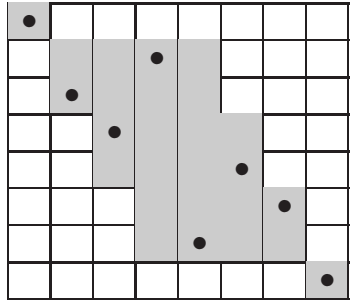
\begin{tabularx}{\textwidth}{|p{0.2cm}|p{0.2cm}|p{0.2cm}|p{0.2cm}|p{0.2cm}|p{0.2cm}|p{0.2cm}|p{0.2cm}|}
\cline{1-8}
  $\bullet$  \cellcolor[gray]{.8}
& & & & &  & & \\
\cline{1-8}
 & \cellcolor[gray]{.8}&\cellcolor[gray]{.8}
& $\bullet$\cellcolor[gray]{.8}& \cellcolor[gray]{.8}&  & & \\
\cline{1-8}
 & \cellcolor[gray]{.8}$\bullet$&\cellcolor[gray]{.8}  & \cellcolor[gray]{.8}& \cellcolor[gray]{.8}&  & & \\
\cline{1-8}
 & & \cellcolor[gray]{.8}$\bullet$&\cellcolor[gray]{.8}& \cellcolor[gray]{.8}&\cellcolor[gray]{.8} &  &  \\
\cline{1-8}
 & & \cellcolor[gray]{.8}&\cellcolor[gray]{.8} & \cellcolor[gray]{.8}&\cellcolor[gray]{.8} $\bullet$&  & \\
 \cline{1-8}
 & & &\cellcolor[gray]{.8} &\cellcolor[gray]{.8}&\cellcolor[gray]{.8}& \cellcolor[gray]{.8}$\bullet$& \\
 \cline{1-8}
 & & &\cellcolor[gray]{.8} &\cellcolor[gray]{.8}$\bullet$ &\cellcolor[gray]{.8}  & \cellcolor[gray]{.8}& \\
 \cline{1-8}
  & & & & &  & & \cellcolor[gray]{.8}$\bullet$\\
  \cline{1-8}
\end{tabularx}
$$
where the shadowed area is the convex hull of the marked cells in $R_{\overline{\tau}_4}$. It is straightforward to see that the rook arrangement is fixed by $\iota$ and hence symplectic. The corresponding symplectic Dellac configuration via the melt map $\mathbf{m}$ is given by:
$$\qquad \qquad \qquad  \qquad \qquad \qquad 
\begin{tabularx}{\textwidth}{|p{0.2cm}|p{0.2cm}|p{0.2cm}|p{0.2cm}|p{0.2cm}|p{0.2cm}|p{0.2cm}|p{0.2cm}|}
\cline{1-8}
  $\bullet$    &  & &$\bullet$ & &   && \\
\cline{1-8}
 &$\bullet$ & $\bullet$&  & &  & &\\
\cline{1-8}
 & & &  & & $\bullet$ &$\bullet$ &\\
\cline{1-8}
 & & &  &$\bullet$ &  & &$\bullet$\\
\cline{1-8}
\end{tabularx}
$$
\end{example}

\bigskip

\noindent\textit{continue the proof of Theorem~\ref{Thm:Cncount}:}\\
The restriction of the blow map $\mathbf{b}:\DC_{2n}\ra\mathcal{R}_{\leq\tau_{2n}}$ to $\SpDC_{2n}$ gives a map $\mathbf{b}':\SpDC_{2n}\ra\mathcal{SR}_{\leq \tau_{2n}}$. By Lemma~\ref{Lem:bij}, $\mathbf{b}$ is injective with $\text{im}(\mathbf{b})=\vp^{-1}(\wt{W}_{\leq\tau_{2n}}^{\wt{J}})$. It implies that $\mathbf{b}'$ is injective with 
$$\text{im}(\mathbf{b}')=\vp^{-1}(\wt{W}_{\leq\tau_{2n}}^{\wt{J}}\cap \wt{W}^\iota)=\psi^{-1}((\wt{W}_{\leq \tau_{2n}}^{\wt{J}})^\iota)$$
and $\mathbf{m}'\circ \mathbf{b}'=\text{id}$.

By the above argument, the blow map $\mathbf{b}'$ gives a bijection $\SpDC_{2n}\stackrel{\sim}{\longrightarrow} (\wt{W}_{\leq \tau_{2n}}^{\wt{J}})^\iota$, composing with $(\vp')^{-1}$ we get a bijection $\SpDC_{2n}\stackrel{\sim}{\longrightarrow} W_{\leq\overline{\tau}_{2n}}^J$.
\end{proof}

\section{Application to torus fixed points}\label{Sec:3}

We show how the construction in Section~\ref{Sec:1} is related to the study of the torus fixed points in the degenerate flag variety.
\subsection{Schubert varieties}
Let $\s_n\in\mathfrak{S}_{2n}$ be the permutation defined as follows:
\begin{equation}\label{Eq:sigma}
{\s_n}(r)=\left\{\begin{matrix} k,& r=2k;\\
n+1+r, & r=2k+1.\end{matrix}\right.
\end{equation}
We see that $\sigma_n$ can be obtained by restricting $\tau_{n+1} \in S_{2n+2}$ to the set $\{2, \ldots, 2n+1\}$. \\
We denote $X_{\sigma_n}$ the Schubert variety corresponding to $\sigma_n$ in the projective variety $SL_n/P$ where $P$ is the standard parabolic subalgebra defined as the stabilizer of the highest weight line of weight $\varpi_1+\varpi_3+\cdots+\varpi_{2n-1}$.  The maximal torus $T_{2n-1}$ of $\text{SL}_{2n}$ acts naturally on $X_{\s_n}$: let $X_{\s_n}^{T_{2n-1}}$ be the set of torus fixed points.\\

It is a standard result that the torus fixed points $X_{\s_n}^{T_{2n-1}}$ can be identified with the quotient $W_{\leq\s_n}^J$ where $W=\mathfrak{S}_{2n}$ and $J=\{2,4,\cdots,2n-2\}$: for $\tau\in W_{\leq\s_n}^J$, the corresponding torus fixed point in $X_{\s_n}^{T_{2n-1}}$ is:
$$\langle e_{\tau(1)}\rangle_{\bc} \subset\langle e_{\tau(1)},e_{\tau(2)},e_{\tau(3)}\rangle_{\bc} \subset\cdots\subset \langle e_{\tau(1)},e_{\tau(2)},\cdots,e_{\tau(2n-1)}\rangle_{\bc} \in X_{\s_n}$$
where $e_1,e_2,\cdots,e_{2n}$ is a fixed basis of $\mc^{2n}$.
\par

\subsection{Degenerate flag varieties}
We fix a basis $\{f_1,f_2,\cdots,f_{n+1}\}$ of $\mc^{n+1}$. Let $\ff l_{n+1}^a$ be the degenerate flag variety of $\text{SL}_{n+1}$ (see \cite{Fei11} for details):
$$\ff l_{n+1}^a=\{(V_1,V_2,\cdots,V_n)\in\prod_{i=1}^n \text{Gr}_{i}(\mc^{n+1})|\ \text{pr}_{i+1}(V_i)\subset V_{i+1}\text{ for any }i=1,\cdots,n\},$$
where $\text{pr}_i:\mc^{n+1}\ra\mc^{n+1}$ is the linear projection along the line generated by $f_i$. By \cite{CFR12}, the torus $T_{2n-1}$ acts on $\ff l_{n+1}^a$: let $(\ff l_{n+1}^a)^{T_{2n-1}}$ be the corresponding set of torus fixed points.

\bigskip

In \cite{CL15}, it is shown that there exists a $T_{2n-1}$-equivariant isomorphism of projective varieties $\zeta: \ff l_{n+1}^a\stackrel{\sim}{\longrightarrow}X_{\s_n}\subset \text{SL}_{2n}/P$. We are especially interested in the image of torus fixed points under $\zeta$:\\
 Fix a basis $\{e_1,e_2,\cdots,e_{2n}\}$ of $\mc^{2n}$. For any $i=1,2,\cdots,n$, we denote the coordinate subspace $U_{n+i}=\langle e_1,e_2,\cdots,e_{n+i}\rangle\subset W$. The surjection $\pi_i:U_{n+i}\ra\mc^{n+1}$ is defined by:
\begin{equation}
{\pi_i}(e_k)=\left\{\begin{matrix} 0& \text{\ if\ }1\leq k\leq i-1;\\
f_k &  \text{\ if\ }i\leq k\leq n+1;\\
f_{k-n-1} &  \text{\ if\ }n+2\leq k\leq n+i.\end{matrix}\right.
\end{equation}
Define $\zeta_i:\text{Gr}_i(\mc^{n+1})\ra \text{Gr}_{2i-1}(\mc^{2n})$ to be the concatenation of the following maps:
$$\text{Gr}_i(\mc^{n+1})\ra \text{Gr}_{2i-1}(U_{n+i})\ra \text{Gr}_{2i-1}(\mc^{2n}),\ \ U\mapsto \pi_i^{-1}(U)\mapsto \pi_i^{-1}(U).$$
Then $\zeta:\ff l_{n+1}^a\ra X_{\s_n}$ is given by $\prod_{i=1}^n \zeta_i$ (see Section 2 of \cite{CL15} for details).
\par
It is clear that the torus $T_{n}$ of $\text{SL}_{n+1}$ acts naturally on $\ff l_{n+1}^a$. By results in Section 7.2 of \cite{CFR12}, any $T_{2n-1}$ fixed point  in $\ff l_{n+1}^a$ is in fact a  $T_{n}$-fixed point. In  \cite{Fei11}, an explicit bijection $\mathbf{f}$ between the $T_{2n-1}$-fixed points and Dellac configuration is provided.

\subsection{A commutative diagram}

As a summary, starting with a $T_{n}$-fixed point in $\ff l_{n+1}^a$, there are two ways to obtain a Dellac configuration:
\begin{enumerate}
\item via the bijection $\mathbf{f}$ given by \cite{Fei11};
\item consider this fixed point as a fixed point in the Schubert variety $X_{\s_n}$, hence identify it with an element in $W_{\leq\s_n}^J$, then melt the corresponding rook arrangement to get a Dellac configuration.
\end{enumerate}
It is natural to ask  whether the following diagram commutes:
\[
\xymatrix{
(\ff l_{n+1}^a)^{T_{2n-1}}=(\ff l_{n+1}^a)^{T_{n}} \ar[r]^-{\mathbf{f}} \ar[d]^-{\beta} & {\DC}_{n+1} \ar[d]^-{\mathbf{b}}\\ X_{\tau_{n+1}}^{T_{2n+1}}=X_{\s_n}^{T_{2n-1}}& W_{\leq\tau_{n+1}}^J \ar[l]_-{\alpha}}.
\]
where the map $\alpha$ is given as follows:\\ for $\s\in W_{\leq\tau_{n+1}}^J$ where $W=\mathfrak{S}_{2n+2}$, we define the map $\alpha$ as follows: $\alpha(\s)$ is the sequence of subspaces $W_1\subset W_2\subset\cdots\subset W_n$ such that $W_i$ is the subspace of $\mc^{2n}$ generated by $e_{\overline{\s}(1)}, e_{\overline{\s}(2)},\cdots,e_{\overline{\s}(2i-1)}$, where $\overline{\sigma}$ is the (well-defined) restriction of $\sigma$ to $\mathfrak{S}_{2n}$. We can identify this element in $X_{\s_n}^{T_{2n-1}}$ with $n$ subsets $J_1,\cdots,J_n$ of $\{1,2,\cdots,2n\}$ such that $J_i=\{\overline{\s}(1),\overline{\s}(2),\cdots,\overline{\s}(2i-1)\}$.
\par
It remains to consider restriction of the map $\zeta$ to fixed points. Here we have to include an extra twist, since the definition of the degenerate flag variety is slightly different in \cite{Fei11} and  \cite{CL15}: let $(V_1,V_2,\cdots,V_n)\in (\ff l_{n+1}^a)^{T_{n}}$, it can be identified (\cite{Fei11}, Corollary 2.11) with $n$ subsets $I_1,I_2,\cdots,I_n$ of $\{1,2,\cdots,n+1\}$ such that $\# I_k=k$ and for any $k=1,2,\cdots,n$, $I_k\backslash\{k+1\}\subset I_{k+1}$.
\par
We denote $\kappa=(12\cdots n+1)^{-1}$ be the inverse of the longest cycle in $\mathfrak{S}_{n+1}$. Suppose that $I_l=\{i_{l,1},i_{l,2},\cdots,i_{l,l}\}$, we denote $I_l^\kappa=\{\kappa(i_{l,1}),\kappa(i_{l,2}),\cdots,\kappa(i_{l,l})\}$. We define a map $p_l:\{1,2,\cdots,n+l\}\ra\{1,2,\cdots,n+1\}$ by
\begin{equation}
{p_l}(s)=\left\{\begin{matrix} 0& \text{\ if\ }1\leq s\leq l-1;\\
s &  \text{\ if\ }l\leq k\leq n+1;\\
s-n-1 &  \text{\ if\ }n+2\leq k\leq n+l.\end{matrix}\right.
\end{equation}
Then $\beta((I_1,I_2,\cdots,I_n))=(T_1,T_2,\cdots,T_n)$ where $T_l=p_l^{-1}(I_l^\kappa)$.

\begin{theorem}
The diagram above commutes, i.e., $\zeta=\alpha\circ \mathbf{b}\circ \mathbf{f}$.
\end{theorem}
The proof is given by a case-by-case examination, we will only give a sketch.

\begin{proof}
We pick $\mathbf{I}=(I_1,I_2,\cdots,I_n)\in(\ff l_{n+1}^a)^{T_{n+1}}$. Recall that the map $\mathbf{f}$ is given in \cite[Proposition 3.1]{Fei11}.
\begin{enumerate}
\item Suppose that $l\notin I_{l-1}$, then $I_l\backslash I_{l-1}=\{j\}$. We consider the case $j>l$: in the Dellac configuration $f(\mathbf{I})$, the cells $(l,l)$ and $(l,j)$ are marked. Then by definition, $\s=\mathbf{b}(f(\mathbf{I}))$ satisfies $\s(2l-1)=l$ and $\s(2l)=j$. Hence in $\alpha(\s)$, $J_l\backslash J_{l-1}=\{l-1,j-1\}$.
\par
We compute $\beta(\mathbf{I})$: it is clear that $I_l^\kappa\backslash I_{l-1}^\kappa=\{j-1\}$, then $p_l^{-1}(I_l^\kappa)\backslash p_{l-1}^{-1}(I_{l-1}^\kappa)=p_l^{-1}(\{l-1,j-1\})=\{l-1,j-1\}$. Therefore $T_l\backslash T_{l-1}=\{l-1,j-1\}$, i.e., $J_l=T_l$.
\par
It is similar to deal with the case $j<l$.
\item Suppose that $l\in I_{l-1}$ and $l\in I_l$, then $I_l\backslash I_{l-1}=\{j\}$. We study the case $j<l$: in the corresponding Dellac configuration, the cells $(l,l+n+1)$ and $(l,j+n+1)$ are marked. The associated permutation $\s=\mathbf{b}(f(\mathbf{I}))$ satisfies $\s(2l-1)=j+n+1$ and $\s(2l)=l+n+1$. Hence in $\alpha(\s)$, $J_l\backslash J_{l-1}=\{j+n,l+n\}$.
\par
For $\beta(\mathbf{I})$: $l\in I_{l-1}\cap I_l$ and $I_l\backslash I_{l-1}=\{j\}$ imply that $l-1\in I_{l-1}^\kappa\cap I_l^\kappa$ and $I_l^\kappa\backslash I_{l-1}^\kappa=\{\kappa(j)\}$. Notice that no matter $j=1$ or $j>1$, $p_l^{-1}(\kappa(j))=j+n$. By the assumption $j<l$, 
$$p_l^{-1}(I_l^\kappa)\backslash p_{l-1}^{-1}(I_{l-1}^\kappa)=p_l^{-1}(\{l-1,\kappa(j)\})=\{j+n,l+n\},$$
which proved $J_l=T_l$.
\par
The case where $j>l$ can be similarly proved.
\item Suppose that $l\in I_{l-1}$ and $l\notin I_l$, then there exists $j_1$ and $j_2$ such that $I_l\backslash I_{l-1}=\{j_1,j_2\}$. We assume that $j_1<l$ and $j_2>l$, in the corresponding Dellac configuration, the cells $(l,j_1+n+1)$ and $(l,j_2)$ are marked, hence in $\alpha(\mathbf{b}(f(\mathbf{I})))$, $J_l\backslash J_{l-1}=\{j_1+n,j_2-1\}$.
\par
For $\beta(\mathbf{I})$, we have
$$p_l^{-1}(I_l^\kappa)\backslash p_{l-1}^{-1}(I_{l-1}^\kappa)=p_l^{-1}(\{\kappa(j_1),j_2-1\})=\{j_1+n,j_2-1\},$$
therefore $J_l=T_l$.
\par
All other cases can be proved in the same way.
\end{enumerate}
\end{proof}

\begin{remark}
A similar diagram without the map $\mathbf{f}$ exists in the symplectic case by changing
\begin{enumerate}
\item the degenerate flag variety to the symplectic degenerate flag variety (see \cite{FFiL12});
\item the Schubert variety of $\text{SL}_{2n}$ by the Schubert variety in the symplectic group (see \cite{CL15});
\item the Dellac configuration by the symplectic Dellac configuration;
\item the set $W_{\leq\tau_{n+1}}^J$ by $W_{\leq \overline{\tau}_{2n+2}}^J$.
\end{enumerate}
\end{remark}

\begin{remark}\label{Rmk:proof}
The original proof of Theorem~\ref{Thm:Ancount} is given by showing the composition $\alpha^{-1}\circ \beta\circ \mathbf{f}^{-1}$ is a bijection: $\mathbf{f}$ is a bijection is shown in \cite{Fei11}; by the main theorem of \cite{CL15}, $\beta$ is a bijection; $\alpha$ is a well-known bijection. Our proof of the theorem uses an intuitive map $\mathbf{b}$ to avoid the geometrical proof.
\end{remark}

\end{document}